\documentclass[12pt]{amsart}
\usepackage[utf8]{inputenc}
\usepackage[T1]{fontenc}
\usepackage[all]{xy}
\usepackage{lipsum}
\usepackage{url}
\usepackage{tikz}
\usepackage{stackrel}
\usepackage{color}

\usepackage{amsmath,amsthm,amssymb}

\newcommand{\aspas}[1]{``{#1}''}

\usepackage{xcolor}

\usepackage{graphicx}

\theoremstyle{definition}

\newtheorem{theorem}{Theorem}[section]

\newtheorem{example}[theorem]{Example}

\newtheorem{proposition}[theorem]{Proposition}
\newtheorem{definition}[theorem]{Definition}

\newtheorem{remark}[theorem]{Remark}

\numberwithin{equation}{section}

\begin{document}

\renewcommand{\bf}{\bfseries}
\renewcommand{\sc}{\scshape}

\newcommand{\TC}{\text{TC}}

\title[The \'{e}tale sectional number is either 1 or infinity]%
{The \'{e}tale sectional number is either 1 or infinity}
\author{Cesar A. Ipanaque Zapata}
\address{Departamento de Matem\'atica, IME-Universidade de S\~ao Paulo, Rua do Mat\~ao, 1010 CEP: 05508-090, S\~ao Paulo, SP, Brazil}
\curraddr{School of Mathematics, University of Minnesota,  
127 Vincent Hall, 206 Church St. SE, Minneapolis MN 55455, USA}
\email{cesarzapata@usp.br}



\begin{abstract} In this work, we show that the \'{e}tale sectional number (\text{\'{E}tale-sec$(-)$}), i.e., the sectional number in the category of topological spaces with the \'{e}tale quasi Grothendieck topology (as defined in arXiv:2410.22515), is either 1 or infinity. Specifically, given a continuous map $f:X\to Y$, we demonstrate that \[\text{\'{E}tale-sec$(f)$}=\begin{cases}
     1,&\hbox{ if  $f$ is locally sectionable,}\\
     \infty,&\hbox{ if  $f$ is not locally sectionable.}
    \end{cases} \]  Additionally, for a path-connected space $X$, the \'{e}tale topological complexity satisfies \[\text{TC}_{\text{\'{e}tale}}(X)=\begin{cases}
     1,&\hbox{ if  $X$ is locally contractible,}\\
     \infty,&\hbox{ if  $X$ is not locally contractible.}
    \end{cases} \] These results provide a way to understand the \aspas{complexity} of maps and spaces within the context of the  \'{e}tale quasi Grothendieck topology, a structure that considers local behavior of maps and spaces. The classification into values of 1 or infinity reflects a dichotomy in the local geometric structure of the map or space, with the presence or absence of local sections or contractibility significantly influencing the outcome.
\end{abstract}

\maketitle



\section{Introduction}\label{secintro}
In this article, a \aspas{space} refers to a topological space, and a \aspas{map} always denotes a continuous map. Fibrations are considered in the Hurewicz sense. Moreover, the symbol \aspas{$\hookrightarrow$} denotes the inclusion map.

\medskip Consider a map $f:X\to Y$. Given a subset $A$ of $Y$, we say that a map $s:A\to X$ is a local section of $f$ if $f\circ s=\mathrm{incl}_A$, where $\mathrm{incl}_A:A\hookrightarrow Y$ is the inclusion map. The usual \textit{sectional number} $\text{sec}(f)$ is the least integer~$m$ such that $Y$ can be covered by $m$ open subsets, each of which admits a local section of~$f$. We set $\text{sec}(f)=\infty$ if no such $m$ exists \cite{schwarz1966}, \cite{berstein1961}. 

\medskip Recently, in \cite{zapata2024}, the author establishes the foundations of sectional theory (\cite{schwarz1966}, \cite{berstein1961}) within category theory. This is achieved by introducing a numerical invariant that measures the minimal number of certain local sections of a morphism in a category with a topology. 

\medskip In this work, for a map $f:X\to Y$, we study the sectional number of $f$ in the category of spaces with the \'{e}tale quasi Grothendieck topology. This numerical invariant is called the \'{e}tale sectional number of $f$, denoted by $\text{\'{E}tale-sec$(f)$}$ (see Definition~\ref{defn:etale-sec-number}).

\medskip In more detail, an \textit{\'{e}tale map} refers a local homeomorphism. That is, $\varphi:A\to Y$ is an \'{e}tale map if and only if for each point $a\in A$, there exists an open subset $O$ of $A$ such that $a\in O$, $\varphi(O)$ is an open subset of $Y$, and the restriction map $\varphi_{\mid}:O\to \varphi(O)$ is a homeomorphism (cf. \cite[p. 88]{maclane1992}).

\medskip The \textit{\'{e}tale sectional number} of a map $f:X\to Y$, denoted by $\text{\'{E}tale-sec$(f)$}$, is the least positive integer $m$ such that there exist \'{e}tale maps $\varphi_1:A_1\to Y,\ldots,\varphi_m:A_m\to Y$ such that $Y=\varphi_1(A_1)\cup\cdots\cup \varphi_m(A_m)$ and for each $i=1,\ldots,m$, there exists a map $\sigma_i:A_i\to X$ satisfying $f\circ\sigma_i=\varphi_i$ (see Definition~\ref{defn:etale-sec-number}). We set $\text{\'{E}tale-sec$(f)$}=\infty$ if no such integer $m$ exists. 

\medskip The main results of this paper are as follows:

\medskip\textbf{Theorem 2.2}  Let $f:X\to Y$ be a map. \begin{enumerate}
     \item[(1)] If $\text{\'{E}tale-sec$(f)$}<\infty$, then \[\text{\'{E}tale-sec$(f)$}=1.\]
     \item[(2)] We have \[\text{\'{E}tale-sec$(f)$}=\begin{cases}
     1,&\hbox{ if  $f$ is locally sectionable,}\\
     \infty,&\hbox{ if  $f$ is not locally sectionable.}
    \end{cases} \] 
 \end{enumerate} 
 
 \medskip In addition, we have the following result for the \'{e}tale topological complexity.

\medskip\textbf{Theorem 2.5} Let $X$ be a path-connected space. We have \[\text{TC}_{\text{\'{e}tale}}(X)=\begin{cases}
     1,&\hbox{ if  $X$ is locally contractible,}\\
     \infty,&\hbox{ if  $X$ is not locally contractible.}
    \end{cases} \] 
 
Motivated by these results, we propose the following question.

\medskip\textbf{Open question:} Is there a null-homotopic surjective \'{e}tale map $\varphi:A\twoheadrightarrow S^n$ over the $n$-sphere ($n\geq 2$), with total space $A$ being Hausdorff and path-connected? 


\section{\'{E}tale sectional number}\label{sec:etaleTC}
 In this section, we recall the notion of \'{e}tale sectional number and prove that this invariant is either one or infinity.  We adopt the notation used in \cite{zapata2024}. 
 
\medskip  An \textit{\'{e}tale map} is a local homeomorphism (in particular, it is an open map), meaning that $\varphi:A\to Y$ is an \'{e}tale map if and only if for each point $a\in A$, there exists an open subset $O$ of $A$ such that $a\in O$, $\varphi(O)$ is an open subset of $Y$, and the restriction map $\varphi_{\mid}:O\to \varphi(O)$ is a homeomorphism (cf. \cite[p. 88]{maclane1992}). Note that if $\varphi:A\to Y$ is an \'{e}tale map and $U$ is an open subset of $A$, then the restriction map $\varphi_{\mid U}:U\to Y$ is also an \'{e}tale map. 

 \medskip In \cite{zapata2024}, the author introduced the notion of \'{e}tale sectional number.

\begin{definition}[\'{E}tale sectional number]\label{defn:etale-sec-number}
 The \textit{\'{e}tale sectional number} of a map $f:X\to Y$ is, denoted by $\text{\'{E}tale-sec$(f)$}$, is the least positive integer $m$ such that there exist \'{e}tale maps $\varphi_1:A_1\to Y,\ldots,\varphi_m:A_m\to Y$ such that $Y=\varphi_1(A_1)\cup\cdots\cup \varphi_m(A_m)$, and for each $i=1,\ldots,m$, there exists a map $\sigma_i:A_i\to X$ satisfying $f\circ\sigma_i=\varphi_i$, that is, the following diagram commutes:
\begin{eqnarray*}
\xymatrix{ X \ar[rr]^{f} & &Y  \\
        &  A_i\ar@{-->}[lu]^{\sigma_i}\ar@{-->}[ru]_{\varphi_i}^{\text{\'{e}t}} & } 
\end{eqnarray*}            
\end{definition}

\medskip From Definition~\ref{defn:etale-sec-number}, note that $\text{\'{E}tale-sec$(f)$}=1$ if and only if there exists a surjective \'{e}tale map $\varphi:A\twoheadrightarrow Y$ (in \cite[Definition 2.9, p. 840]{carchedi2013}, the author refers to it as an \textit{\'{e}tale cover} of $X$) together with a map $\sigma:A\to X$ satisfying $f\circ\sigma=\varphi$.

\medskip We say that $f:X\to Y$ is \textit{locally sectionable} if,  for each point $y\in Y$, there exists an open subset $U$ of $Y$  containing $y$ such that there exists a local section $\sigma:U\to X$ of $f$. Equivalently, there exists an open covering  $\{U_\lambda\}_{\lambda\in\Lambda}$ of $Y$ such that for each $\lambda$, there exists a local section $\sigma_\lambda:U_\lambda\to X$ of $f$. Note that if $f$ is locally sectionable, then $f$ is surjective.

\medskip Note that if $\text{\'{E}tale-sec$(f)$}<\infty$, then $f$ is locally sectionable. We will show that the reverse implication holds. In contrast, if $f$ is locally sectionable, it does not imply that the sectional number $\text{sec}(f)<\infty$ in general. 

\medskip We establish and prove our first main result.

\begin{theorem}\label{thm:main-theorem}
 Let $f:X\to Y$ be a map. \begin{enumerate}
     \item[(1)] If $\text{\'{E}tale-sec$(f)$}<\infty$, then \[\text{\'{E}tale-sec$(f)$}=1.\]
     \item[(2)] We have \[\text{\'{E}tale-sec$(f)$}=\begin{cases}
     1,&\hbox{ if  $f$ is locally sectionable,}\\
     \infty,&\hbox{ if  $f$ is not locally sectionable.}
    \end{cases} \] 
 \end{enumerate} 
\end{theorem}
\begin{proof}
\noindent\begin{enumerate} 
 \item[(1)] Suppose that $\text{\'{E}tale-sec$(f)$}=m<\infty$. Let $\varphi_1:A_1\to Y,\ldots,\varphi_m:A_m\to Y$ be \'{e}tale maps such that $Y=\varphi_1(A_1)\cup\cdots\cup \varphi_m(A_m)$, and for each $i=1,\ldots,m$, there exists a map $\sigma_i:A_i\to X$ satisfying $f\circ\sigma_i=\varphi_i$. Take the topological sum $A=\coprod_{i=1}^{m}A_{i}\times\{i\}$, the map $\varphi:\coprod_{i=1}^{m}A_{i}\times\{i\}\to Y$ by $\varphi(x,i)=\varphi_i(x)$, and the map $\sigma:\coprod_{i=1}^{m}A_{i}\to X$ by $\sigma(x,i)=\sigma_i(x)$ (i.e., $\varphi_{| A_i}=\varphi_i$ and $\sigma_{| A_i}=\sigma_i$ for each $i$). Note that $\varphi$ is surjective (here we use that $Y=\bigcup_{i=1}^{m}\varphi_i(A_i)$), it is an \'{e}tale map (recall that each map $\varphi_i:A_i\to Y$ is an \'{e}tale map), and $f\circ\sigma=\varphi$ (here we use that $f\circ\sigma_i=\varphi_i$ for each $i$). Therefore, $\text{\'{E}tale-sec$(f)$}=1$. 

 \item[(2)]  Suppose that $f$ is locally sectionable, i.e., there exists an open covering  $\{U_\lambda\}_{\lambda\in\Lambda}$ of $Y$ such that for each $\lambda$, there exists a local section $\sigma_\lambda:U_\lambda\to X$ of $f$. Take the topological sum $A=\coprod_{\lambda\in\Lambda}U_{\lambda}\times\{\lambda\}$, the map $\varphi:\coprod_{\lambda\in\Lambda}U_{\lambda}\times\{\lambda\}\to Y$ given by $\varphi(x,\lambda)=x$, and the map $\sigma:\coprod_{\lambda\in\Lambda}U_{\lambda}\times\{\lambda\}\to X$ given by $\sigma(x,\lambda)=\sigma_\lambda(x)$ (i.e., $\varphi_{| U_\lambda}=\text{incl}_{U_\lambda}$ and $\sigma_{| U_\lambda}=\sigma_\lambda$ for each $\lambda\in\Lambda$). Note that $\varphi$ is surjective (here we use that $Y=\bigcup_{\lambda\in\Lambda}U_\lambda$), it is an \'{e}tale map (recall that each inclusion $\text{incl}_{U_\lambda}:U_\lambda\hookrightarrow Y$ is an \'{e}tale map), and satisfies the equality $f\circ\sigma=\varphi$ (here we use that each $\sigma_\lambda$ is a local section of $f$). Therefore, $\text{\'{E}tale-sec$(f)$}=1$.
 
 Now, suppose $f$ is not locally sectionable. We will show that $\text{\'{E}tale-sec$(f)$}=\infty$. We proceed by contradiction. Suppose that $\text{\'{E}tale-sec$(f)$}<\infty$. Then, by Item (1)$, \text{\'{E}tale-sec$(f)$}=1$. Let $\varphi:A\twoheadrightarrow Y$ be a surjective \'{e}tale map,  together with a map $\sigma:A\to X$ satisfying $f\circ\sigma=\varphi$. For each $y\in Y$, let $a\in A$ be such that $\varphi(a)=y$ (here we use that $\varphi$ is surjective). Since $\varphi$ is an \'{e}tale map, there exists an open subset $O$ of $A$ such that $a\in O$, $U_y=\varphi(O)$ is an open subset of $Y$, and the restriction map $\varphi_|:O\to \varphi(O)$ is a homeomorphism. Note that $y\in \varphi(O)=U_y$, and \[\sigma\circ \text{incl}_{O}\circ (\varphi_|)^{-1}\] is a local section of $f$, where $\text{incl}_{\varphi(O)}:\varphi(O)\hookrightarrow Y$ and $\text{incl}_{O}:O\hookrightarrow A$ are the inclusion maps, and $(\varphi_|)^{-1}$ is the inverse of $\varphi_|:O\to \varphi(O)$. Hence, we obtain an open covering $\{U_y\}_{y\in Y}$ of $Y$ such that for each $y$, there exists a local section $\sigma_y:U_y\to X$ of $f$. Therefore,  $f$ is locally sectionable, which contradicts the assumption that $f$ is not locally sectionable. 
  \end{enumerate} 
\end{proof}

\medskip Let $X$ be a space. Let $X^{[0,1]}=\{\gamma:[0,1]\to X|~ \text{ $\gamma$ is a continuous path}\}$ be the space of all paths in $X$, topologized with the compact-open topology. Consider the fibration $e_2^X:X^{[0,1]}\to X\times X$ given by \[e_2^X(\gamma)=(\gamma(0),\gamma(1)).\] Given a point $x_0\in X$, consider the subspace $P\left(X,x_0\right)=\{\gamma\in X^{[0,1]}:~\gamma(0)=x_0\}$. Note that $P\left(X,x_0\right)$ is contractible. Consider the fibration $e_1^X:P\left(X,x_0\right)\to X$ given by \[e_1^X(\gamma)=\gamma(1).\] 

\medskip The following remark says that the Cartesian product of \'{e}tale maps is an \'{e}tale map. 

\begin{remark}\label{rem:product-etale}
 Let $f:X\to Y$ and $g:Z\to W$ be \'{e}tale maps. The Cartesian product $f\times g:X\times Z\to Y\times W$ is an \'{e}tale map. Indeed, let $(x,z)\in X\times Z$ and consider open subsets $U\subset X$ and $V\subset Z$ with $x\in U$ and $z\in V$, such that $f(U)\subset Y$ and $g(V)\subset W$ are open subsets and $f_\mid:U\to f(U)$ and $g_\mid:V\to g(V)$ are homeomorphisms. Note that the restriction of $f\times g$ over $U\times V$ is given by the Cartesian product $f_\mid\times g_\mid:U\times V\to f(U)\times g(V)$. 
\end{remark}

\medskip We have the following statement.

\begin{proposition}\label{prop:e2x-e1x}
    Let $X$ be a space. We have the following: \begin{enumerate}
        \item[(1)] $\text{\'{E}tale-sec$(e_2^X)$}=1$ if and only if $\text{\'{E}tale-sec$(e_1^X)$}=1$.
        \item[(2)] $\text{\'{E}tale-sec$(e_2^X)$}=\infty$ if and only if $\text{\'{E}tale-sec$(e_1^X)$}=\infty$. 
    \end{enumerate} 
\end{proposition}
\begin{proof}
\noindent\begin{enumerate}
    \item[(1)]  Let $x_0\in X$. The following diagram 
\begin{eqnarray*}
\xymatrix{  P\left(X,x_0\right) \ar@{^{(}->}[r]^{ } \ar[d]_{e_1^X} & X^{[0,1]} \ar[d]^{e_2^X}  \\ 
       X  \ar[r]_{i_0} &  X\times X }
\end{eqnarray*} is a pullback, where $i_0(x)=(x_0,x)$. If $\text{\'{E}tale-sec$(e_2^X)$}=1$, then $\text{\'{E}tale-sec$(e_1^X)$}=1$ (see \cite[Theorem 3.25(1), p. 35]{zapata2024}).  

Now, suppose that $\text{\'{E}tale-sec$(e_1^X)$}=1$. Let $\varphi:A\twoheadrightarrow X$ be a surjective \'{e}tale map together with a map $\sigma:A\to P\left(X,x_0\right)$ satisfying $e_1^X\circ\sigma=\varphi$. Note that the map $\varphi\times \varphi:A\times A\twoheadrightarrow X\times X$ is surjective, and by Remark~\ref{rem:product-etale}, it is also an \'{e}tale map. The map $s:A\times A\to X^{[0,1]}$ given by \[s(a,a')(t)=\begin{cases}
    \sigma(a,1-2t),&\hbox{ if $0\leq t\leq 1/2$,}\\
    \sigma(a',2t-1),&\hbox{ if $1/2\leq t\leq 1$,}
\end{cases}\] satisfies the equality $e_2^X\circ s=\varphi\times \varphi$. Therefore, $\text{\'{E}tale-sec$(e_2^X)$}=1$.
\item[(2)] It follows from Theorem~\ref{thm:main-theorem} together with Item (1). 
\end{enumerate}
\end{proof}

\medskip A space $X$ is called \textit{locally contractible} or \textit{locally categorical} if for each $x\in X$, there exists an open subset $U$ of $X$ such that $U$ is \textit{contractible in $X$} or \textit{categorical} \cite[p. 331]{james1978}, i.e., the inclusion map $\text{incl}_U:U\hookrightarrow X$ is null-homotopic (\cite[p. 215]{farber2003}, cf. \cite[Theorem A.7, p. 525]{hatcher2002}). Equivalently, for each open subset $V$ of $X$ and each point $x\in V$, there exists an open subset $U$ of $X$ such that $x\in U\subset V$ and the inclusion map $\text{incl}_U:U\hookrightarrow X$ is null-homotopic. Similarly, there exists an open covering $\{U_\lambda\}_{\lambda\in\Lambda}$ of $X$ such that each inclusion map $\text{incl}_{U_\lambda}:U_\lambda\hookrightarrow X$ is null-homotopic. Such a covering $\{U_\lambda\}_{\lambda\in\Lambda}$ is called \textit{categorical} \cite[p. 331]{james1978}.

\medskip Note that if $e_1^X$ is locally sectionable, then $X$ is locally contractible. The other implication holds whenever $X$ is path-connected. 

\medskip In \cite[Example 3.11]{zapata2024}, the invariant $\text{\'{E}tale-sec$(e_2^X)$}$ is called the \textit{\'{e}tale topological complexity} of $X$ and is denoted by \[\text{TC}_{\text{\'{e}tale}}(X)=\text{\'{E}tale-sec$(e_2^X)$}.\]

\medskip Theorem~\ref{thm:main-theorem} together with Proposition~\ref{prop:e2x-e1x} imply the following result for the \'{e}tale topological complexity.

\begin{theorem}\label{cor:e1x-locally-contrac}
Let $X$ be a path-connected space. We have \[\text{TC}_{\text{\'{e}tale}}(X)=\begin{cases}
     1,&\hbox{ if  $X$ is locally contractible,}\\
     \infty,&\hbox{ if  $X$ is not locally contractible.}
    \end{cases} \]      
\end{theorem}

We present the following example.
 
\begin{example}
  Let $X$ be the shrinking wedge circles, i.e., it is the subspace $X\subset\mathbb{R}^2$ that is the union of the circles $C_n$ of radius $1/n$ and center $(1/n,0)$ for $n=1,2,\ldots$ (see \cite[Example 1.25, p. 49]{hatcher2002}). Note that $X$ is path-connected and is not locally contractible (because the point $(0,0)\in X$ does not have an open subset $U$ of $X$ such that the inclusion map $\text{incl}_U:U\hookrightarrow X$ is null-homotopic, see \cite[p. 215]{farber2003}). Hence, by Theorem~\ref{cor:e1x-locally-contrac}, \[\text{TC}_{\text{\'{e}tale}}(X)=\infty.\] 
  \end{example}

\medskip We have the following remark.

\begin{remark}\label{rem:nul-etale-locally-contrac}
 Given a path-connected space $X$ and $x_0\in X$: \begin{enumerate}
     \item[(1)] Suppose that $X$ is locally contractible, i.e., there exists an open covering $\{U_\lambda\}_{\lambda\in\Lambda}$ of $X$ such that each inclusion map $\text{incl}_{U_\lambda}:U_\lambda\hookrightarrow X$ is null-homotopic. Since $X$ is path-connected, we can suppose that each inclusion map $\text{incl}_{U_\lambda}:U_\lambda\hookrightarrow X$ is homotopic to the constant map $\overline{x}_0:U_\lambda\hookrightarrow X$, that is, there exists a null-homotopy $H^\lambda:U_\lambda\times [0,1]\to X$ such that $H^\lambda_0=\text{incl}_{U_\lambda}$ and $H^\lambda_1=\overline{x}_0$. 
 
 Similarly, from the proof of Theorem~\ref{thm:main-theorem}(2), we take the topological sum $A=\coprod_{\lambda\in\Lambda}U_\lambda\times\{\lambda\}$ together with the surjective \'{e}tale map $\varphi:A\to X$ given by $\varphi_{| U_\lambda}=\text{incl}_{U_\lambda}$. Furthermore, $\varphi$ is homotopic to the constant map $\overline{x}_0:A\to X$ because we have the null-homotopy $H:A\times [0,1]\to X$ given by $H_{| U_\lambda\times [0,1]}=H^\lambda$. Therefore, we obtain a null-homotopic surjective \'{e}tale map \[\varphi:A\twoheadrightarrow X.\]
 \item[(2)] From Item (1), note that $A$ is Hausdorff whenever $X$ is Hausdorff. However, the total space $A$ is not necessarily path-connected.     
 \end{enumerate}      
\end{remark}

On the other hand, we note that the universal covering map $q:\mathbb{R}\to S^1, r\mapsto q(r)=e^{2\pi i r}$, is a null-homotopic surjective \'{e}tale map over the $1$-sphere $S^1$, with the total space being path-connected.

\begin{remark}\label{rem:for higher-spheres}
 For $n\geq 2$, it is well known that there is no null-homotopic covering map on the $n$-dimensional sphere $S^n$. 
\end{remark}

However, the following example presents a null-homotopic \'{e}tale cover of  $S^n$.   

\begin{example}\label{exam:nulhomo-etale-cover-sphere}
Let $n\geq 1$. Fix $p=(p_1,\ldots,p_{n+1})$ and $q=(q_1,\ldots,q_{n+1})$, two distinct points of $S^n$, with $p_{n+1},q_{n+1}<0$. Consider the quotient space \[A=\left(S^n\setminus\{p\}\sqcup S^n\setminus\{q\}\right)/\sim,\] where the equivalence relation $\sim$ on the disjoint union of two copies of punctured spheres identifies every point $x=(x_1,\ldots,x_{n+1})$ of the upper hemisphere (i.e., $x_{n+1}>0$) of the first copy with the corresponding point of the upper hemisphere of the second copy. The two copies of a point $x=(x_1,\ldots,x_n,0)$ of the equator are not identified,  and they do not have any disjoint neighborhoods, so $A$ is not Hausdorff. Note that $A$ is contractible, and the natural map $\varphi:A\to S^n$ is a  surjective \'{e}tale map (i.e., it is an \'{e}tale cover of  $S^n$). Therefore, we obtain a null-homotopic \'{e}tale cover of  $S^n$.  
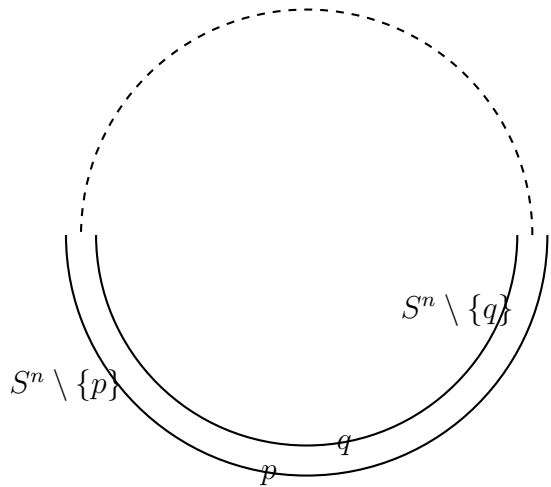
\[
\begin{tikzpicture}
\node at (-0.5,-3.2) {$p$};
\node at (0.5,-2.8) {$q$};

\draw[thick, dashed] (3,0) arc[start angle=0, end angle=180, radius=3];
\draw[thick] (-3.2,0) arc[start angle=180, end angle=360, radius=3.2];
\draw[thick] (-2.8,0) arc[start angle=180, end angle=360, radius=2.8];

\node at (-3.2, -2) {$S^n \setminus \{p\}$};
\node at (2, -1) {$S^n \setminus \{q\}$};

\node at (0, -4) {Quotient space $A=\left(S^n\setminus\{p\}\sqcup S^n\setminus\{q\}\right)/\sim$};

\end{tikzpicture}\]
\end{example}

Motivated by Remark~\ref{rem:nul-etale-locally-contrac} and Example~\ref{exam:nulhomo-etale-cover-sphere}, we propose the following question.

\medskip\textbf{Open question:} Is there a null-homotopic surjective \'{e}tale map $\varphi:A\twoheadrightarrow S^n$ over the $n$-sphere ($n\geq 2$), with total space $A$ being Hausdorff and path-connected? 

\medskip In light of \cite[Proposition 1, p. 381; Proposition 6, p. 389]{manfredo2016} together with Remark~\ref{rem:for higher-spheres}, $A$ must be non-compact or $\varphi$ must not satisfy the property of lifting arcs.

\section*{Acknowledgements}
The author would like to thank grant\#2023/16525-7, and grant\#2022/16695-7, S\~{a}o Paulo Research Foundation (FAPESP) for financial support.

\bibliographystyle{plain}

\end{document}